\documentclass[twocolumn]{autart}

\usepackage{amsmath}
\usepackage{amssymb}
\usepackage{algorithm}
\usepackage{algorithmic}
\usepackage{graphicx} % Required for inserting images
\usepackage{natbib}
\usepackage{xcolor}

\usepackage[algo2e,ruled,vlined,linesnumbered,norelsize]{algorithm2e}
\DeclareMathOperator*{\argmin}{argmin}
\def\X{\mathbb{X}}
\def\U{\mathbb{U}}
\def\V{\mathbb{V}}
\def\W{\mathbb{W}}
\def\Z{\mathbb{Z}}
\def\D{\mathbb{D}}
\def\R{\mathbb{R}}

\def\I{\mathcal{I}}
\def\S{\mathcal{S}}
\def\Vset{\mathcal{V}}
\def\False{\textnormal{False}}
\def\True{\textnormal{True}}
\def\feasible{\mathit{feasible}}

\allowdisplaybreaks[4]
\usepackage{parskip}

\begin{document}

\begin{frontmatter}
\title{Robust Multi-step Model Predictive Control with Guaranteed Stability\thanksref{t1}}
\thanks[t1]{This work was supported by the EPSRC Doctoral Training Partnership EP/W524311/1, project reference 2743399.}
\author[Oxford]{Sebastian Steffen},
\author[Oxford]{Mark Cannon}
% \thanks[ead1]{Email: sebastian.steffen@eng.ox.ac.uk.}
% \thanks[ead2]{Email: mark.cannon@eng.ox.ac.uk.}
\address[Oxford]{Department of Engineering Science, University of Oxford, UK.\\Email: sebastian.steffen@eng.ox.ac.uk, mark.cannon@eng.ox.ac.uk}

%\thanks[fundinginfo]{This work was supported by the EPSRC Doctoral Training Partnership EP/W524311/1, project reference 2743399.}

\begin{abstract}
We present a method of ensuring recursive feasibility and input-to-state stability of robust nonlinear Model Predictive Control (MPC) with multi-step predictors.
Although feasibility guarantees are well-established for the case of single-step models applied recursively over a finite horizon, such guarantees are missing in naive MPC formulations that use distinct multi-step models to predict the system state at different future points in time. This issue arises because of potential inconsistencies in multi-step predictions generated at different times.
%
%Although feasibility guarantees are well-established for the case of single-step models that are applied recursively over a finite horizon, such guarantees are missing for multi-step predictors that use distinct models to predict the system state at different future points in time. This issue arises because of potential inconsistencies in multi-step predictions generated at different times.
%
Our approach performs an \textit{a priori} sufficient feasibility check of the robust nonlinear MPC optimisation problem, and uses information from previous solutions to provide a fallback based on previously certified prediction sets.
We illustrate the proposed predictor-substitution strategy with a simple numerical example.%
%, given earlier information on the model state. 
%We present an algorithm using a predictor-substitution strategy to ensure feasibility recursively.% 
\end{abstract}
\end{frontmatter}
% \begin{abstract}
% We present a mechanism for ensuring recursive feasibility and stability of nonlinear tube MPC using multi-step predictors in a standard receding horizon strategy, without the need for multi-rate formulations. Methods to guarantee these properties are well established for standard tube MPC, where state predictions over a finite horizon are made using recursive applications of a one-step-ahead plant model. However, the multi-step case lacks these guarantees due to the fact that distinct functions are used to predict the same future tube cross-sections at different solution steps, causing potential inconsistencies in the predicted tubes under the same input sequences. 
% %
% Our approach checks for potential infeasibility of the nonlinear tube MPC problem \textit{a priori}, and uses information from previous solutions to fall back on previously verified sets, given earlier state information. We present an algorithm using a predictor-substitution strategy to guarantee construction of a feasible problem. 
% \end{abstract}

\section{Introduction}
Model Predictive Control (MPC) is a widely used strategy for control problems involving state and input constraints, external disturbances and model uncertainty~\citep{mayne:2024,kouvaritakis:2016}. At each discrete time step, a constrained receding horizon optimal control problem (RHOCP) is solved to find the control sequence that minimises performance predicted over a finite horizon using a model. 
%Nonlinear MPC has found applications across a range of industrial processes and engineering systems due to the potential for nonlinear models to accurately predict system evolution~\cite{qin:2000,mayne:2024}. 
%Nonlinear models present additional challenges over linear MPC however, including increased online computational load since the RHOCP is in general a nonconvex nonlinear program (NLP), and the need for more sophisticated tools to ensure robustness by propagating the effects of disturbances and model uncertainty through nonlinear dynamics.
%However, for nonlinear models, the optimal control problem is generally a nonlinear problem (NLP), rather  convex, and therefore there is no guarantee that it can be solved with a fixed amount of time.
%
%A large body of work has investigated various ways to address these issues. A common approach reformulates the optimal control problem as a sequential convex program (SCP), with the NLP typically approximated using local linearisations around candidate trajectories~\cite{diehl:2005}. The convexified problem (often a QP) is then iterated by linearising using trajectories generated using previous solutions \cite{mao:2017, lishkova:2025}. These 
Robust MPC requires the effects of model errors and uncertainty to be propagated across the future prediction horizon, and this can present additional challenges for the case of nonlinear model dynamics~\citep{diehl:2005,mao:2017,lishkova:2025}.
However, the use of distinct functions for each time step of the prediction horizon allows a simpler formulation of the RHOCP, since multi-step predictors can account for model errors and uncertainty at individual future time steps rather than bounding their accumulation over a horizon. 

This is particularly important in the case of models learnt from data.
%
%An additional advantage of multi-step models emerges when they are derived from data rather than analytically. 
Recent comparisons of single-step and multi-step linear models show that when models are misspecified due to partial observability, learnt multi-step models show lower compounding error \citep{somalwar:2025} and better data-efficiency during learning \citep{somalwar:2026}. Similarly, nonlinear multi-step predictors derived from data using deep neural networks can provide lower average prediction error than recursive application of a single-step model~\citep{steffen:2025}.

% \begin{figure}[t]
% \centerline{\includegraphics[width=\linewidth]{figures/visual_abstract_v3.pdf}}
% \caption{Comparison of predictor substitution method for recovering recursive feasibility in multi-step MPC despite inconsistent tube sections across optimisation steps, compared to standard tube MPC.}
% \label{fig:vis_abstract}
% \end{figure}
\begin{figure}[t]
\centerline{\includegraphics[scale=0.5, trim={31mm 29mm 150mm 20mm}, clip]{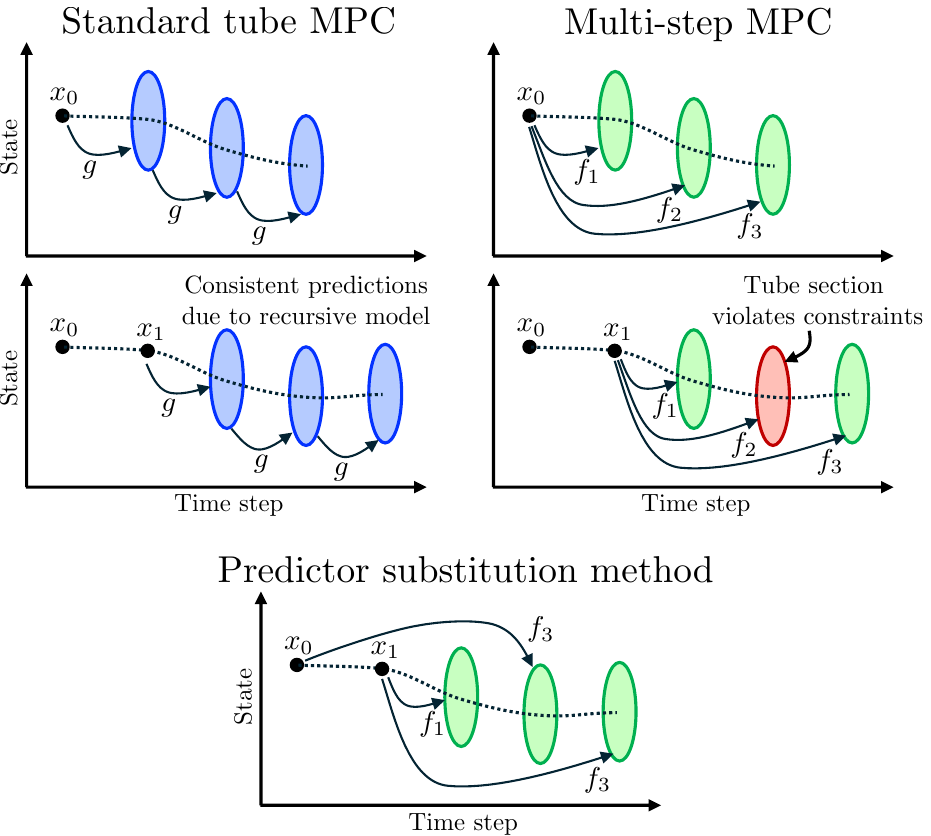}}%
\caption{Predictor substitution method for recovering recursive feasibility in multi-step MPC despite inconsistent tube sections across optimisation steps, and standard tube MPC.}
\label{fig:vis_abstract}
\end{figure}

The use of linear multi-step predictors in MPC is well-established.
%A growing body of literature is emerging on MPC using linear multi-step predictors. 
\citet{terzi:2022} propose a linear multi-rate MPC scheme using an autoregressive model with exogenous input (ARX) structure. 
%The approach features a multi-rate structure in which the prediction horizon is partitioned into blocks corresponding to the multi-rate models, 
The approach partitions the prediction horizon into blocks corresponding to multi-rate models, and in contrast to standard receding-horizon MPC, an entire block of inputs is applied before recomputation. \citet{saccani:2023} extend this multi-rate approach to state-space models with parametric uncertainty and combine it with homothetic tubes to ensure robust constraint satisfaction. Multi-step models are identified from data in~\citet{balim:2024} 
%via an intermediate state space model 
and used in stochastic MPC with chance constraints. Likewise~\citet{fiedler:2023} consider stochastic MPC, but employ a probabilistic multi-step identification method incorporating implicit state estimation. The approach is applied to nonlinear systems by treating the effect of nonlinearity as a disturbance.
%by interpreting nonlinear effects as part of the additive process noise.

A smaller body of work has directly considered the nonlinear case. \citet{maddalena:2021} use kernel-based methods to learn multi-step nonlinear mappings from data and provide deterministic finite-horizon safety guarantees, but do not establish recursive feasibility.
\citet{deJong:2024} establish recursive feasibility using an interpolated initial predictor state in a lifted Koopman space. However, this construction does not directly enforce hard constraints on plant states or outputs.
To the best of our knowledge, recursive feasibility guarantees have not previously been established for robust MPC with additive bounded multi-step prediction errors, hard state constraints, and standard single-step receding horizon updates.
%
% However, neither approach can ensure recursive feasibility when hard constraints on system states or outputs are present, and
% %
% to the best of our knowledge, recursive feasibility guarantees have not previously been established for robust MPC with state and output constraints using multi-step predictors while retaining standard single-step receding horizon control updates.
%
%However, no nonlinear MPC schemes with multi-step predictors proposed to date can ensure recursive feasibility of the RHOCP. 
%Moreover, existing multi-step linear MPC schemes with recursive feasibility guarantees rely on multi-rate formulations, and therefore do not use the standard receding-horizon approach of recomputing the control sequence at each sampling step, given the latest state or output measurement. 

%\subsection*{Contribution and Outline}
In this work we propose a method of ensuring recursive feasibility and input-to-state stability of the RHOCP with multi-step predictors. 
At each time step, the method first constructs a feasible candidate solution
%At each time step, a suboptimal feasible RHOCP solution is first identified 
by searching over a set of candidate trajectories, replacing predictors that violate constraints with previously verified alternatives via a predictor substitution strategy (Fig.~\ref{fig:vis_abstract}). Robust feasibility of a terminal condition is ensured through a variable horizon strategy.
%, and we present a method for recovering the original horizon at later time steps.
Our method applies to nonlinear multi-step prediction models with bounded additive disturbances. 

The paper is organised as follows. After the problem formulation in Section~\ref{sec:problem}, Sections~\ref{sec:normal_mpc} and~\ref{sec:multi_step_feasibility} discuss tube MPC with single-step and multi-step models and present the proposed predictor substitution algorithm. Section~\ref{sec:multi-step_rhocp} describes the RHOCP for multi-step MPC with feasibility and asymptotic performance guarantees.
%is discussed in Section~\ref{sec:stability}. 
A numerical example illustrating the algorithm is discussed in Section \ref{sec:example}, and we provide concluding remarks in Section~\ref{sec:conclusion}.
 
% \subsection*{Notation} \label{ssec:notation}
% Vector quantities are expressed in simple lower-case roman or greek font, e.g. $x$, $\eta$. Matrices are denoted in uppercase font, e.g. $A$, and sets are denoted with blackboard bold font, e.g. $\X$. A subscript denotes the quantity at a particular discrete time step, e.g. $x_{k}$, and subscript with two terms separated by a colon denote a sequence of the quantity from the $a$\textsuperscript{th} to the $b$\textsuperscript{th} steps, e.g.  $v_{a:b} = \{v_{a}, \ \dots, \ v_{b}\}$. $\mathbb{A} \oplus \mathbb{B}$ denotes the Minkowski sum, i.e. $\{a + b : a \in \mathbb{A}, b \in \mathbb{B}\}$. The form $||x||^{2}_{Q}$ denotes the matrix-weighted squared vector norm $x^{T}Qx$, and $|\X|$ denotes the cardinality of a set. 

\section{Problem formulation} \label{sec:problem}
%\subsection{System model, multi-step predictors, and control problem}
Consider a system with state $x\in\X\subset\R^{n_x}$, control input $u\in\U\subset\R^{n_u}$, and unknown disturbance $d\in\D\subset\R^{n_x}$
\begin{equation}\label{eq:plant}
    x_{k+1} = g(x_{k}, u_{k}) + d_{k} ,
\end{equation} 
where $k=0,1,\ldots$ is the discrete time index. We assume that the sets $\X$, $\U$ and $\D$ are compact and $g$ is continuous.

The evolution of the state of \eqref{eq:plant} over a prediction horizon of $N$ steps can be represented using multi-step models as 
\begin{equation} \label{eq:multi-step_dynamics}
x_{i|k} \in \{f_{i}(x_{k}, \, u_{0:i-1|k})\} \oplus \W_i, \quad i = 1, \, \dots, \, N, 
\end{equation} 
where $f_i:\X\times\R^{n_u i}\to\R^{n_x}$,
%where $f_i:\X\times\U^i\to\R^{n_x}$,
%
$\oplus$ denotes Minkowski addition,
$u_{0:i-1|k}=\{u_{0|k},\ldots,u_{i-1|k}\}$, 
with $x_{i|k},u_{i|k}$ denoting the $i$ steps ahead values of $x$ and $u$ predicted at time $k$, 
and $\W_{i}\subset \R^{n_x}$ is a disturbance set accounting for the effects of model errors and disturbances $d_k,\ldots,d_{k+i-1}$ in~\eqref{eq:plant}.
We assume $\W_{1},\ldots,\W_N$ are  known, compact sets that over-bound the multi-step model uncertainty. 
% The evolution of the state of \eqref{eq:plant} over a prediction horizon of $N$ steps can be represented using $N$ multi-step models as 
% % further assume that the model~(\ref{eq:plant}) is approximated for all $k$ using multi-step predictors over a horizon of $N$ steps 
% \begin{equation} \label{eq:multi-step_dynamics}
% x_{i|k} = f_{i}(x_{k}, \, u_{0:i-1|k}) + w_{k}^{(i)}, \quad i = 1, \, \dots, \, N, 
% \end{equation} 
% where $u_{0:i-1|k}=\{u_{0|k},\ldots,u_{i-1|k}\}$, with $x_{i|k},u_{i|k}$ denoting the $i$ steps ahead values of $x$ and $u$ predicted at time $k$,
% and $w^{(i)}_{k} \in \W_{i}$ $i=1,\ldots,N$ are unknown disturbances accounting for model errors and the effects of disturbances $d_k,\ldots,d_{k+i-1}$ in~\eqref{eq:plant}. We assume the multi-step disturbance sets
% %$w_k^{(i)}\in \W_{i}$, where 
% $\W_{1},\ldots,\W_N$ are known but possibly conservative, convex, and compact. 

We write $u$ as a perturbation $v$ on a continuous feedback law $\kappa(x)$%
\begin{equation} \label{eq:control_law}%
   u_{k} = \kappa(x_{k}) +  v_{k} .
\end{equation}
where $v\in\V$ and $\V$ is compact since $(x,u)\in\X\times\U$. 
For each $i$, assume that a closed-loop multi-step predictor $\bar{f}_i:\X\times\R^{n_u i}\to\R^{n_x}$ is known such that, under $u_{j|k} = \kappa(x_{j|k}) + v_{j|k}$ for $j=0,\ldots,i-1$, we have 
%$x_{i|k}\in\X_i(x_k,v_{0:i-1|k})$, where
%Then $f_i$ in \eqref{eq:multi-step_dynamics} can be replaced with
% \begin{equation} \label{eq:fb_transformation}
% x_{i|k} \in \{\bar{f}_{i}(x_{k}, v_{0:i-1|k})\} \oplus \W_i ,
% % \bar{f}_{i}(x_{k}, v_{0:i-1|k}) = f_{i}(x_{k}, u_{0:i-1|k}),
% \end{equation}
%with $v_{0:i-1|k}=\{v_{0|k},\ldots,v_{i-1|k}\}$. A set containing all states predicted by (\ref{eq:fb_transformation}) with initial state $x_k$ is denoted
%and control parameterisation (\ref{eq:control_law}) is denoted
%with perturbation $v_{0:i-1|k}$ is denoted
\begin{equation} \label{eq:tube_section}
x_{i|k} \in \X_i(x_k, v_{0:i-1|k}) = \{\bar{f}_{i}(x_{k}, v_{0:i-1|k})\} \oplus \W_{i} ,
\end{equation}
where $v_{0:i-1|k}=\{v_{0|k},\ldots,v_{i-1|k}\}$ and $\X_0(x_k) = \{ x_{k}\}$.
To simplify notation we 
%sometimes omit the dependence on $(x_k,v_{0:i-1|k})$ and 
write $\X_i(x_k, v_{0:i-1|k})$ as $\X_{i|k}$,
and we denote the set $\{f(x): x\in\X_{i|k}\}$ for any function $f(\cdot)$ as $f(\X_{i|k})$.
%Denote the set $\{\kappa(x) : x\in\X_{i|k}\}$ as $\kappa(\X_{i|k})$.
%
The predicted set sequence forms a tube:
\[
\X_{0:N|k}= \{ \X_{0|k}, \, \dots, \, \X_{N|k}\}.
\]
The following assumption formalises the requirement that the multi-step model in (\ref{eq:tube_section}) robustly bounds the state of the system~\eqref{eq:plant} for all $k$ and  $i\in\{0,\ldots,N\}$.
% We rely on the assumption that the multi-step model \eqref{eq:multi-step_dynamics} is a correct approximation of the the true nonlinear plant \eqref{eq:plant}, i.e. that the predicted sets contain the true value of the system state under the same input trajectory. 

\begin{assum} \label{as:correct_predictor_sets}
For every $k\geq 0$, $x_k\in\X$, $i \in\{1,\ldots,N\}$, $v_{0:i-1|k}\in\V^i$, and $d_k,\ldots,d_{k+i-1}\in\D$, the state $x_{k+i}$ generated by~(\ref{eq:plant}) with $u_{k+j} = \kappa(x_{k+j}) + v_{j|k}$ for $j=0,\ldots,i-1$, satisfies
$x_{k+i} \in \X_{i} (x_{k},v_{0:i-1|k})$.
%
%For any given $v_{k-N},\ldots,v_{k-1}$, and for all disturbances $d_{k-N}\ldots, d_{k-1}\in\D$, the state $x_{k}$ of (\ref{eq:plant}) is consistent with the multi-step predictors~(\ref{eq:tube_section}), i.e., for all $k\geqslant N$ and all $i\in\{1,\ldots,N\}$,
%\begin{equation}
%    x_{k} \in \X_{i} (x_{k-i},v_{k-i:k-1}).
%    \end{equation}
\end{assum}

% \vspace{-\baselineskip}
% \textcolor{blue}{Include a remark here on how the sets $\W_i$ might be determined from data, by computing empirical bounds based on the error in multi-step predictions relative to samples, perhaps using PAC bounds derived from an assumption on the distribution of errors.}

\begin{rem}
Assumption~\ref{as:correct_predictor_sets} requires $\W_i$ to be a deterministic outer bound on the $i$-step prediction error $\smash{e_i}=\smash{x_i} - \smash{\bar{f}_i(z)}$ for all $z=(x_0,v_{0:i-1})\in\Z_i=\X\times\smash{\V^i}$.
Approximate error bounds can be computed using errors evaluated on a finite sample set $\smash{z^{(1)}},\ldots,\smash{z^{(n_s)}}\in\Z_i$ with a level of confidence that depends on $n_s$.
A robust bound $\W_i$ can be computed given knowledge of $g$ and $\D$ in model~(\ref{eq:plant}) by propagating the disturbances through an incremental linear difference inclusion, as discussed in Appendix~\ref{appendix:disturbance_sets}.
%
%using error samples $\smash{e_i^{(1)}},\ldots,\smash{e_i^{(m)}}$ if a Lipschitz bound $L_i$ on the prediction error map in is known. If the samples $\smash{z_i^{(1)}},\ldots,\smash{z_i^{(m)}}$ have covering radius $\rho_i$ in $\Z_i$, then a deterministic outer approximation is $\smash{\W_i} = \mathrm{co}\{\smash{e_i^{(1)},\ldots,e_i^{(m)}}\}\oplus \smash{L_i\rho_i\B}$, where $\B$ is the unit ball in $\R^{n_x}$.
%
%Finite error samples alone do not establish such a bound. One possible construction uses a Lipschitz bound on the  $i$-step prediction error map $\smash{e_i=x_i - \bar{f}_i(x_0,v_{0:i-1})}$ together with a covering radius of the sampled state-control perturbation domain. 
%Alternatively, $\W_i$ may be computed by propagating the known disturbance and model error bounds using a linear difference inclusion as discussed in Appendix A.1.
\end{rem}

% \begin{rem}
% The set $\W_i$ is required to provide deterministic bounds on prediction error $\smash{e_i=x_i - \bar{f}_i(x_0,v_{0:i-1})}$. These bounds can be computed using samples of $e_i$ if Lipschitz bounds on $\bar{f}_i$ are available, or using knowledge of $\D$ in~(\ref{eq:plant}) and a linear difference inclusion bound on $g$ if such is available (e.g.~see the example of Appendix~\ref{appendix:disturbance_sets}). 
% %
% %Alternatively, given an independent calibration sample of prediction errors $e_i$, an order-statistic rule could be used to choose a probabilistic bound $\W_i$ on future prediction errors with a prescribed level of confidence.
% %and they are less likely to be conservative. 
% \end{rem}

% The control objective is to minimise an upper bound over model uncertainty of a quadratic cost 
% \begin{equation}\label{eq:infinite_horizon_cost}
%     \sum^{\infty}_{k=0}(\| 
%     x_{k}\|^{2}_{Q} + \|u_{k}\|^{2}_{R})
% \end{equation}
% where $\|x\|_Q^2$ is the quadratic form $x^\top Q x$, and  $Q,R$ are positive definite weighting matrices, denoted $Q,R \succ 0$.

The control objective is to satisfy the state and input constraints robustly while maintaining a small quadratic stage cost. The RHOCP uses a finite-horizon worst-case cost, and Section~\ref{sec:multi-step_rhocp} establishes an upper bound on the asymptotic time average of the closed-loop stage cost

\section{Tube MPC with single-step predictors} \label{sec:normal_mpc}
To motivate the development for multi-step predictors, we consider a generic robust tube MPC based on a single-step model with an RHOCP employing a worst-case cost
\begin{equation} \label{eq:normal_mpc_cost}
J(v_{0:N-1|k}, \X_{0:N|k}) =  \sum_{i=0}^{N-1}\max_{x\in\X_{i|k}}\ell(x, v_{i|k}) 
+ \max_{x\in\X_{N|k}} \ell_T(x)
%       + \max_{x\in\X_{N|k}} \|x\|^{2}_{P}
\end{equation} 
where the terminal cost is given by $\ell_T(x) = \|x\|_P^2=x^\top P x$ for $P\succ 0$, and $\ell(x,v)$ is a quadratic stage cost given by
\begin{equation}
    \ell(x, v) = \|x\|^{2}_{Q} + \|\kappa(x) + v\|^{2}_{R} 
\end{equation}
for positive definite matrices $Q,R \succ 0$, and where $\X_{0:N|k} = \{\X_{0|k},\ldots,\X_{N|k}\}$ denotes a tube generated by recursive application of the single-step model~(\ref{eq:plant}), with
\begin{equation}\label{eq:normal_mpc_tube}
%\X_{i+1|k} \supseteq g(\X_{i|k}, \kappa(\X_{i|k})+v_{i|k}) \oplus \D
\X_{i+1|k} \supseteq \{g(x, \kappa(x)+v_{i|k})\} \oplus \D
\quad \forall x\in\X_{i|k}, 
\end{equation}
for $i=0,\ldots,N-1$, and $x_k\in \X_{0|k}$.
%(where $g(\X_{i|k},\U_{i|k})=\{g(x,u) : (x,u)\in\X_{i|k}\times\U_{i|k}\}$).
The RHOCP solved at the $k$th discrete time step is given by
%Let $v^*_{0:N-1|k}$ denote the solution of the RHOCP
\begin{subequations}\label{eq:RHOCP}
\begin{equation}
(v^*_{0:N-1|k}, \X^\ast_{0:N|k}) = \argmin_{v_{0:N-1|k}, \X_{0:N|k}}  J(v_{0:N-1|k}, \X_{0:N|k})
\end{equation}
subject to $x_k\in\X_{0|k}$, and, for $i = 0, \dots, N-1$, (\ref{eq:normal_mpc_tube}) and
\begin{align}
\X_{i|k} &\subseteq\X 
%, & & i = 0, \dots, N-1
\label{eq:state_constraint}\\
\kappa(\X_{i|k}) \oplus \{v_{i|k}\} &\subseteq \U 
%, & & i = 0, \dots, N-1
%\kappa(x) + v_{i|k} &\in \U \quad \forall  x\in\X_{i|k}, & & i = 0, \dots, N-1
\label{eq:input_constraint}\\
\X_{N|k} &\subseteq \X_{T} . 
\label{eq:terminal_constraint}
\end{align}    
\end{subequations} 
The control input at the $k$th time step is $u_k=\kappa(x_k) + v^\ast_{0|k}$, and the 
RHOCP is solved for all $k\geqslant 0$. 
%RHOCP is solved again at the next time step. 
%
The terminal set $\X_{T}$ is robust positively invariant (RPI) so that
\begin{equation}\label{eq:single_RPI_condition}
%g\bigl(\X_T, \kappa(\X_T)\bigr) \oplus \D \subseteq \X_{T}\subseteq\X
%\ \text{and} \ 
%\kappa(\X_T) \subseteq \U 
\{g(x, \kappa(x))\} \oplus \D \subseteq \X_{T}\subseteq\X 
\ \text{and} \ 
\kappa(x) \in \U ,
\end{equation}
for all $x\in\smash{\X_T}$. 
We assume that 
that all tube cross-sections use the same set parameterisation, so that a cross-section from a previous solution can be copied into any stage of the tube at the current time step.
%the parameterisation of $\X_{0:N-1|k}$ is consistent with itself, so that it permits $\X_{i|k} = \X_{j|k}$ for any $i,j\in\{0,\ldots,N\}$.
%The parameterisation of $\X_{0:N-1|k}$ must be consistent with itself and $\X_T$, i.e.\ it should permit $\X_{i|k} = \X_{j|k}$ and $\X_{i|k} = \X_{T}$ for any $i,j\in\{0,\ldots,N\}$.

Recursive feasibility is the property that if the RHOCP is feasible at time $k$, then it will also be feasible at $k+1$. This property applies to (\ref{eq:RHOCP}) due to the recursive definition of the tube~(\ref{eq:normal_mpc_tube}) and the terminal constraint \eqref{eq:terminal_constraint}. 
In particular, at $k+1$ the shifted sequence 
\begin{equation}\label{eq:tail_sequence}
\tilde{v}_{0:N-1|k+1} = \{ v^\ast_{1:N-1|k}, \, 0\}  %\tilde{\X}_{0:N|k+1}=\{\X_{1:N|k}^\ast,\X_T\}
\end{equation}
satisfies the RHOCP constraints since ${x_{k+1}\in\smash{\X^\ast_{1|k}}}$ implies  (\ref{eq:normal_mpc_tube})-(\ref{eq:RHOCP}b,c) hold with
$\smash{\X_{i|k+1}}=\smash{\X_{i}(x_{k+1},v^\ast_{1:i|k})} \subseteq \smash{\X_{i+1|k}^\ast}$, 
% and (\ref{eq:RHOCP}d) holds with
% $\X_{N-1|k+1}=\X_{N-1}(x_{k+1},v^\ast_{1:N-1|k}) \subseteq \X_{N|k}^\ast$, so that
% $\X_{N|k+1}\subset\X_T$ due to (\ref{eq:single_RPI_condition}).
and~(\ref{eq:single_RPI_condition}) ensures $\smash{\X_{N|k+1}}\subseteq\X_T$
since $\smash{\X_{N-1|k+1}}=$ $\smash{\X_{N-1}(x_{k+1},v^\ast_{1:N-1|k})} \subseteq \smash{\X_{N|k}^\ast} \subseteq \smash{\X_T}$,
so (\ref{eq:RHOCP}d) holds.

%Therefore at $k+1$, the time-shifted previous solution, $\tilde{v}_{0:N-1|k+1} = \{ v^{*}_{1:N-1|k}, \, 0\}$, $\smash[t]{\tilde{\X}_{0:N|k-1}}=\{\X_{1:N|k},\X_T\}$ provides a feasible solution to the RHOCP.

Closed-loop stability of this generic tube MPC formulation follows from a Lyapunov-like property of the optimal RHOCP objective 
%a cost-decrease condition of the optimal value of the RHOCP across 
at consecutive time steps. To ensure robust stability, $P\succ 0$ in the terminal cost $\ell_T(x) = \|x\|^{2}_{P}$ is chosen to satisfy %
\begin{equation} \label{eq:single_terminal_condition}%
\ell_T\bigl(g(x, \kappa(x)) + d\bigr) - \ell_T(x) \leqslant {-\ell(x, 0)} + \sigma^{2}
%    \|x\|^{2}_{P} - \|g(x, \kappa(x)) + d\|^{2}_{P} \geqslant l(x, 0) - \sigma^{2}
\end{equation} 
for all $x\in \X_{T}$, for all $d \in \D$, for some constant $\sigma$. 
Under mild conditions (e.g.~$\X_T$, $\D$ are compact and $g$, $\kappa$ are Lipschitz continuous for $(x,u)\in\X_T\times\U$), this ensures that the MPC law $u_k = \kappa(x_k) + \smash{v^\ast_{0|k}}$ 
%shifted control sequence $\tilde{v}_{0:N-1|k}$ applied for all $k\geqslant 1$ 
yields a bounded cost $\smash{J_{k}^\ast}$, that the asymptotic time-average stage cost is bounded by $\sigma^2$ (implying a limit on the impact of disturbances on performance), and that the origin, $x=0$, of (\ref{eq:plant}) is input-to-state practically stable (ISpS)~\cite[Definition 6]{limon:2009} (implying that the system state converges to a neighbourhood of $x=0$ dependent on the size of the disturbance set $\D$).

%a no greater than that at the previous step, minus the stage cost $\ell$ at step $k$, plus some positive scalar $\sigma^{2}$ which determines the allowable impact of disturbances. This cost decrease property yields a finite bound on the average stage cost. 

\section{Multi-step predictors} \label{sec:multi_step_feasibility}
To incorporate multi-step models in the RHOCP \eqref{eq:RHOCP} we use the multi-step predictors~(\ref{eq:tube_section}) to construct a tube bounding the future states of the model~(\ref{eq:plant}) instead of the recursive single-step  conditions~(\ref{eq:normal_mpc_tube}).
The terminal set $\X_T$ and terminal cost $\ell_T$ are now assumed to satisfy RPI and Lyapunov-like conditions analogous to (\ref{eq:single_RPI_condition}) and (\ref{eq:single_terminal_condition}) based on~(\ref{eq:tube_section}) with $i=1$.

\begin{assum}\label{as:terminal_ingredients}
Terminal set $\X_T$ and cost $\ell_T(x)$ satisfy
\begin{gather}
\label{eq:terminal_cost}
\ell_T\bigl(\bar{f}_1(x, 0) + w\bigr) - \ell_T(x) \leqslant {-\ell(x, 0)} + \sigma^{2}
%\|x\|^{2}_{P} - \|f_{1}(x, \kappa(x)) + w\|^{2}_{P} \succeq \|x\|^{2}_{Q}+\|\kappa(x)\|^{2}_{R} - \sigma^{2}
\\
\bar{f}_1(x, 0) + w \in \X_{T}\subseteq\X \ \ \text{and} \ \ 
\kappa(\X_T) \subseteq \U 
% \{\bar{f}_1(x, 0)\} \oplus \W_{1} \subseteq \X_{T}\subseteq\X \ \ \text{and} \ \ 
% \kappa(x) \subseteq \U 
\label{eq:multi-step_terminal}
%\ \ \forall x\in \X_{T}.
\end{gather}
for all $(x,w)\in\X_T\times\W_1$, some $\sigma$, and $\ell_T(x) = \|x\|_P^2$, $P\succ 0$.
\end{assum}

%With multi-step predictors
With this change however,
the sets $\X_{i-1}(x_{k+1},v_{1:i-1|k})$ and $\X_{i}(x_k,v_{0:i-1|k})$ that bound the state of (\ref{eq:plant}) at time $i+k$ are no longer nested in general.
This potential inconsistency is due to the use of distinct predictors $\bar{f}_i,\bar{f}_{i-1}$ and disturbance sets $\W_{i},\W_{i-1}$. As a result, RHOCP feasibility at time $k$ no longer ensures feasibility at $k+1$.
%, as illustrated in Fig.~\ref{fig:state_trajectories}. 
In this section we address this issue using \textit{a priori} feasibility checks and predictor substitutions.\\[-1ex]

% \begin{figure}[htb]
%     \centering
%     \includegraphics[width=0.9\linewidth]{figures/state_diagram_updated.pdf}
%     \caption{Illustration of inconsistent predicted state bounds across time steps}
%     \label{fig:state_trajectories}
% \end{figure}

\subsection{Checking constraint satisfaction \textit{a priori}} \label{ssec:feasibility_check}
A computationally inexpensive sufficient condition for RHOCP feasibility can be checked  by evaluating the state and input trajectories generated by a prescribed set of control-perturbation sequences.
Let $\Vset_k$ be a finite set of candidate perturbation sequences $v_{0:N-1|k}\in\V^N$ at time $k$ 
and let $\I(x,v_{0:N-1},N)$ denote, for given $v_{0:N-1}$, initial state $x$, and horizon $N$, the indices of prediction time steps at which state or control constraints are violated by the multi-step prediction models of~(\ref{eq:tube_section}):
\begin{multline} \label{eq:feasibility_check}
\I(x,v_{0:N-1},N) = \bigl\{ i : 1 \leq i \leq N-1 
\\
\text{if }\X_{i}(x,v_{0:i-1}) \nsubseteq \X \text{ or } \kappa(\X_{i}(x,v_{0:i-1})) \oplus \{v_{i}\} \nsubseteq \U\bigr\}, 
\\
\cup \bigl\{N \text{ if } \X_{N}(x,v_{0:N-1}) \nsubseteq \X_T \bigr\}.
\end{multline}
A sufficient condition for RHOCP feasibility at time $k$ is
\begin{equation}\label{eq:feasibility_condition}
\exists \  v_{0:N-1|k}\in \Vset_k \ \text{such that} \ \I(x_k, v_{0:N-1|k},N) = \emptyset.    
\end{equation}
%Although $\Vset_{k}$ is not straightforward to compute explicitly, it is relatively easy to check whether a sequence $v_{0:N-1|k}$ belongs to $\Vset_k$ using multi-step predictors~(\ref{eq:tube_section}).
%for given $\bar{f}_i$ and $\W_i$. 
Given a set $\Vset_k$  containing the shifted sequence $\smash{\tilde{v}_{0:N-1|k}}$, this can be used to ensure recursive feasibility.%\\[-1ex]

%an \textit{a priori} certificate that the RHOCP is feasible.

\subsection{Feasibility recovery and predictor substitution}
Under Assumption \ref{as:correct_predictor_sets}, existence of an RHOCP solution at time $k$ implies that, at time $k+1$, the state of~(\ref{eq:plant}) can be safely steered to the RPI set $\X_T$ over a future horizon of $N$ steps, even if \eqref{eq:feasibility_condition} is violated at time $k+1$ due to inconsistency in multi-step predictors. However, as we show in this section, a feasible problem can be defined at time $k+1$ using the RHOCP solution at time~$k$ to replace some predictors at $k+1$. 

Assume the RHOCP has solution $\smash{(v^\ast_{0:N-1|k},\X^\ast_{0:N|k})}$ at time $k$ and let 
$\smash{\tilde{v}_{0:N-1|k+1}}= \smash{\{v^\ast_{1:N-1|k}, 0 \}}$.
Suppose that $\I_{k+1} = \I(x_{k+1},\tilde{v}_{0:N-1|k+1},N)$ is non-empty. 
Then, for all $i\in\I_{k+1}$ satisfying $i<N$, we can substitute the $i$-step predictor  with initial  state $x_{k+1}$ with an $(i+1)$-step predictor starting from $x_k$.
This simply replaces the set $\X_{i}(x_{k+1},\tilde{v}_{0:i-1|k+1})$ %$=\X_{i}(x_{k+1},v^\ast_{1:i|k})$ 
with $\smash{\X_{i+1|k}^\ast}=\smash{\X_{i+1}(x_k,v^\ast_{0:i|k})}$ and removes $i$ from $\I_{k+1}$ (since the RHOCP is by assumption feasible at time $k$)
without affecting the other elements of $\I_{k+1}$ (since $v^\ast_{0:i|k}=\{v^\ast_{0|k},\tilde{v}_{0:i-1|k+1}\}$).
%
%We define the infeasible predictor indices at $k$ as $\I_{k} = \psi(\X_{i:N|k})$. For all $i \in \I_{k}$, the sets $\X_{i|k}$ can be replaced with $\X_{i+1|k-1}$. 

\begin{table}[tb]
\centering
\caption{Feasibility testing example: \checkmark\ indicates constraint satisfaction, \textsf{X} indicates violation\label{tab:feasible_idxs}}%
\begin{tabular}{c|c@{~~}c@{~~}c@{~~}c@{~~}c@{~~}c@{~~}c@{~~}c}
\hline 
time- & \multicolumn{7}{c}{predicted time step}
\\[-1ex]
step & $k$ & $k+1$ & $k+2$ & $k+3$ & $k+4$& $k+5$ & $k+6$ & $k+7$ 
\\\hline
$k$ 
& \checkmark & \checkmark & \checkmark& \checkmark & \checkmark &&&
\\[-1ex]
$k+1$  
& & \checkmark & \checkmark & \checkmark & \textsf{X} & \checkmark &&
\\[-1ex]
$k+2$ 
& & & \checkmark & \checkmark & \textsf{X}  & \checkmark & \checkmark&  
\\[-1ex]
$k+3$ 
& & &  &\checkmark & \textsf{X}  & \checkmark & \checkmark & \checkmark \\\hline
\end{tabular}%
\end{table}

Predictor substitution can be repeated at later times if $\smash{\I(x_{k+p},\tilde{v}_{0:N-1|k+p},N)}$ is non-empty for $p>1$; thus predictors with initial state $x_k$ may certify feasibility at time $k+p$ for $p< N$.
Table~\ref{tab:feasible_idxs} illustrates a case in which feasibility of the RHOCP with $N=4$ is certified with substitutions at times $k+1$, $k+2$, and $k+3$, respectively:
%$k+p$ for $p=1,2,3$:
\begin{align*}
&\X_3(x_{k+1},\tilde{v}_{0:2|k+1}) \gets \X_4(x_k,v^\ast_{0:3|k}) ,
%\text{ at } k+1,
\\
&\X_2(x_{k+2},\tilde{v}_{0:1|k+2}) \gets \X_4(x_k,\{v^\ast_{0|k},v^\ast_{0:2|k+1}\}) ,
%\text{ at } k+2,
\\
&\X_1(x_{k+3},\tilde{v}_{0|k+3}) \gets \X_4(x_k,\{v^\ast_{0|k},v^\ast_{0|k+1},v^\ast_{0:1|k+2}\}) .
%\text{ at } k + 3.
\end{align*}
%
% Table~\ref{tab:feasible_idxs} illustrates the case that RHOCP feasibility is certified
% at $k+1$
% with $\X_3(x_{k+1},\tilde{v}_{0:2|k+1})$ replaced by $\X_4(x_k,v^\ast_{0:3|k})$,
% at $k+2$ with $\X_2(x_{k+2},\tilde{v}_{0:1|k+2})$ replaced by $\X_4(x_k,\{v^\ast_{0|k},v^\ast_{0:2|k+1}\})$,
% and at $k+3$ with $\X_1(x_{k+3},\tilde{v}_{0|k+3})$ replaced by $\X_4(x_k,\{v^\ast_{0|k},v^\ast_{0|k+1},v^\ast_{0:1|k+2}\})$.
%
% Table~\ref{tab:feasible_idxs} illustrates the case that RHOCP feasibility is certified at $k+p$ for $p=1,2,3$ by substituting $\X_3(x_{k+1},\tilde{v}_{0:2|k+1})$, $\X_2(x_{k+2},\tilde{v}_{0:1|k+2})$ and $\X_1(x_{k+3},\tilde{v}_{0|k+3})$
% respectively with
% $\X_4(x_k,v^\ast_{0:3|k})$
% $\X_4(x_k,\{v^\ast_{0|k},v^\ast_{0:2|k+1}\})$, 
% and $\X_4(x_k,\{v^\ast_{0|k},v^\ast_{0|k+1},v^\ast_{0:1|k+2}\})$
%since $\X^\ast_{4|k}=\X_4(x_k,v^\ast_{0:3|k}) \subseteq \X_T\subseteq \X$.
%
Predictor substitution cannot be used to replace the $N$-step ahead predictor at time $k+1$ because the RHOCP solved at time $k$ only considers future time steps up to time $k+N$. 
%In this situation we therefore rely 
%Hence if $N\in\I(x_k,\tilde{v}_{0:N-1|k})$, we 
Hence for the case that $\X_{N}(x_{k+1},\tilde{v}_{0:N-1|k+1}) \nsubseteq \X_T$, we 
rely on the RPI property (\ref{eq:multi-step_terminal}) of the terminal set $\X_T$ to certify RHOCP feasibility.
One possibility is to use a composite $N$-step predictor combining $(N-1)$-step and $1$-step predictors. However,
to maintain a prediction structure with one multi-step predictor for each future time step, we introduce a variable prediction horizon, $\bar{N}_k$, initialised with $\bar{N}_0 = N$. 
To maximise this horizon, at time $k+1$, $k\geq 0$, we set $\bar{N}_{k+1} \gets N$ and repeatedly apply the update
\begin{equation}\label{eq:horizon_shrink}
\begin{aligned}
\bar{N}_{k+1}&\gets
\max\{1, \bar{N}_{k}-1, \bar{N}_{k+1}-1\} 
\\
&\qquad 
 \text{if} \ \X_{\bar{N}_{k+1}}(x_{k+1},\tilde{v}_{0:N-1|k+1}) \nsubseteq \X_T.
\end{aligned}
\end{equation}
until there is no further change in $\bar{N}_{k+1}$.
This process terminates with either $\smash{\X_{\bar{N}_{k+1}}(x_{k+1},\tilde{v}_{0:N-1|k+1}) \subseteq \X_T}$,
or with $\bar{N}_{k+1} = \bar{N}_{k}-1$ or $1$; in the latter cases the feasibility of the terminal constraint at $k+1$ can be ensured by a predictor substitution since $\smash{\X^\ast_{\bar{N}_{k}|k}\subseteq \X_T}$ by construction.
Since~(\ref{eq:horizon_shrink}) searches for the longest horizon $\bar{N}_{k+1} \leqslant N$ such that terminal constraint feasibility is ensured, this policy reinstates the original horizon $\bar{N}_{k+1} = N$ whenever possible, thus reducing suboptimality by maintaining as many  control degrees of freedom in the RHOCP as possible.

% To retain a prediction structure with one multi-step predictor for each future time step,
% we introduce a variable length prediction horizon $\bar{N}_k$, defined at time $k+1$ by
% \begin{equation}\label{eq:horizon_shrink}
% \bar{N}_{k+1} =\begin{cases}
% \bar{N}_{k} & \text{if } \X_{\bar{N}_{k}}(x_{k+1},\tilde{v}_{0:\bar{N}_{k}-1|k}) \subseteq \X_T
% \\
% \max \{ 1, \bar{N}_{k} - 1 \} & \text{otherwise}
% \end{cases}
% \end{equation}
% with $\bar{N}_0 = N$. 
% Reducing the horizon by one in this way allows terminal constraint feasibility to be certified by a predictor substitution at time $k+1$ since $\X^\ast_{\bar{N}_{k}|k}\subseteq \X_T$ by construction.
%
% If, on the other hand, $\bar{N}_{k+1}=\bar{N}_k$ after applying (\ref{eq:horizon_shrink}) and $\bar{N}_k<N$, then we attempt to recover the original horizon by setting
% \begin{equation}\label{eq:horizon_grow}
% \bar{N}_{k+1} \gets \min\{\bar{N}_{k+1} + 1, N\}
% \end{equation}
% and reapplying (\ref{eq:horizon_shrink}) until $\bar{N}_{k+1}$ can no longer be increased. Increasing $\bar{N}_{k+1}$ in this way reduces RHOCP suboptimality by
% introducing additional optimisation variables and delaying the imposition of the terminal cost and terminal constraint.

To keep track of predictor substitutions at successive time steps, we denote the predictor indices employed in the RHOCP at time $k$ as $\smash{\S_{k}}=\{\smash{ s_{1|k}},\ldots, \smash{ s_{\bar{N}_k|k}}\}$.
Given $\smash{\bar{N}_{k+1}}$ and $\smash{\I_{k+1}} = \I(\smash{ x_{k+1}}, \smash{\tilde{v}_{0:N-1|k+1}}, \smash{\bar{N}_{k+1}})$, 
the update for $\S_{k+1}$ at $k+1$ is
%the update for $\eta_{k+1}$ at $k+1$ is
% %
% We initially set $\eta_{k+1}\gets\{1,\ldots,\bar{N}_{k+1}\}$, and for $i = \bar{N}_{k+1}, \ldots 1$ apply
% \begin{equation}\label{eq:feasible_indices}
% \eta_{i|k+1} \gets \eta_{i+1|k}  
% \quad \text{if} \ i \in  \I_{k+1} \text{ or (\ref{eq:cost_decrease_constraint}) is violated}
% \end{equation}
\begin{equation}\label{eq:feasible_indices}
s_{i|k+1} = \begin{cases}
s_{i+1|k} & \text{if} \ i \in  \I_{k+1},
\\
i & \text{otherwise},
\end{cases}
\quad
i=1,\ldots,\bar{N}_{k+1},
\end{equation}
which we denote for convenience as $\S_{k+1} = \S(\I_{k+1},\S_k)$.
When substituting a predictor at time $k$, the difference between the substituted predictor index and the index that would be used if no substitution were made is $p = s_{i|k}-i$. Hence for given $\S_k$, the predictors can be written, for $i = 1,\dots,\bar{N}_{k}$,
\begin{equation} \label{eq:substituted_tube}
    \X_{i|k} = 
    \bigl\{\bar{f}_{s_{i|k}} (x_{k-p}, \{v_{k-p:k-1},v_{0:i-1|k}\})\bigr\} \oplus \W_{s_{i|k}}
\end{equation}
with $\X_{0|k} = \{ x_{k}\}$ and $v_{k-p:k-1} = \{v^\ast_{0|k-p},\ldots,v^\ast_{0|k-1}\}$.\\[-1ex]

%, $s = s_{i|k}$.

% \subsection{Terminal Constraint and Horizon Shrinking}
% The proposed predictor substitution method \eqref{eq:feasible_indices} cannot be applied for the current terminal step $k+N$, since prior optimisation problems have only considered steps up to $k+N-1$. In this situation, we propose temporarily shrinking the horizon in order to recover the feasibility of the RHOCP. We therefore define the (possibly) shrunk horizon at time step $k$ as $\bar{N}_{k} \in \mathbb{Z}_{+}$. If necessary, this procedure can be applied more than once, such that  $\bar{N}_{k} \in \big\{ \mathrm{max}(1, \bar{N}[k-1]   - 1), \,\bar{N}[k-1] \big\}$. 

%\subsection{Horizon Recovery}

% Unlike the standard formulation, the multi-step RHOCP is not guaranteed satisfy the cost-decrease condition required for stability, due to the possible inconsistencies in the state tubes. Therefore, to ensure stability, we include in the RHOCP a constraint that the objective function, denoted $J_{k}$, satisfies the following condition
% \begin{equation} 
%     J^{*}_{k+1} \leq J^{*}_{k} - \ell(x_{k}, v_{k}) + \sigma ^{2}.
% \end{equation} 

\subsection{Predictor substitution algorithm}
Algorithm~\ref{alg:updates_fn_indices} summarises the procedure
% described in this section 
for ensuring RHOCP feasibility. At times $k >0$ the algorithm computes a horizon $\bar{N}_k$ and predictor indices $\S_k$ such that state and control constraints and the terminal constraint are satisfied for some perturbation sequence $\smash{ v_{0:N-1|k}}$, thus ensuring $\I(\smash{ x_k}, \smash{ v_{0:N-1|k}}, \smash{\bar{N}_k}) =\emptyset$ in~(\ref{eq:feasibility_check}).
The algorithm also imposes a constraint on the predicted cost included in the RHOCP to ensure an asymptotic bound on closed-loop performance:
%with initial state $x_{k}$, control perturbation $\tilde{v}_{0:N-1|k+1}$ and predictor substitutions $\eta_{k+1}$:
\begin{equation}\label{eq:cost_decrease_constraint}
J(v_{0:N-1|k}, x_k, \S_{k}) \leqslant J^\ast_{k-1} - \ell(x_{k-1},v^\ast_{0|k-1}) + \sigma^2.
\end{equation}
Here $J(\cdot)$ is the RHOCP objective 
%predicted cost 
defined in Section~\ref{sec:multi-step_rhocp} and
$\sigma$ satisfies Assumption~\ref{as:terminal_ingredients}.

{\setlength{\algomargin}{1.5em}
\begin{algorithm2e}
\caption{Predictor substitution with horizon shrinking and recovery}
\label{alg:updates_fn_indices}
\DontPrintSemicolon
\SetKwInOut{Input}{Input}
\SetKwInOut{Output}{Output}
\Input{$x_{k}$, $\tilde{v}_{0:N-1|k}$, $J^{*}_{k-1}$, $\bar{N}_{k-1}$, $\S_{k-1}$}
\Output{$\bar{N}_k$, $\S_k$}
% compute list of values $\alpha_{i}\in [0,1]$, $i=1,\ldots,n_{\alpha}$\;
construct a candidate set $\Vset_k$ containing the shifted sequence $\smash{\tilde{v}_{0:N-1|k}}=\{v^\ast_{1:N-1|k-1}, 0\}$\;
%$n^{\mathrm{subs}} \gets \bar{N}_{k-1}-1$\;
%$\I^{\mathrm{min}} \gets \{1, \, \dots, \, \bar{N}_{k-1}-1 \}$\;
$\bar{N} \gets N+1$\;
$\feasible \gets \False$\;
% HORIZON RECOVERY WHILE LOOP
\While{$\feasible = \False$ \textnormal{\textbf{and}} $\bar{N}\geqslant \bar{N}_{k-1}$ \textnormal{\textbf{and}} $\bar{N}>1$}{
    $\bar{N} \gets \bar{N}-1$\;
    $\I^{\min} \gets \{1, \, \dots, \, \bar{N}\}$\;
    % LINE SEARCH OVER CANDIDATE TRAJECTORIES
    % \For{$\alpha \in \{\alpha_{1}, \, \dots, \, \alpha_{n_{\alpha}}\}$}{}
    \For{$v_{0:N-1} \in\Vset_k$}{ 
        % compute $\X_{1:\bar{N}|k}$ given $v_{0:\bar{N}-1}$\;
        compute $\I \gets \I(x_k, v_{0:\bar{N}-1}, \bar{N})$ using \eqref{eq:feasibility_check}\;
        \If{$\{\bar{N}_{k-1}, \dots, \bar{N}\} \cap \I = \emptyset$ \textnormal{\textbf{or}} $\bar{N} < \bar{N}_{k-1}$}{
            compute $\S \gets \S(\I,\S_{k-1})$ using \eqref{eq:feasible_indices}\;
            % HERE WE NEED TO CHECK THE INDICES AGAIN - CONSTRAINT SATISFACTION FOR THE SUBSTITUTED SETS IS ONLY GUARANTEED UNDER THE SHIFTED SEQUENCE, BUT NOT UNDER ANY OF THE OTHER SEQUENCES IN THE SET \Vset
            compute $\X_{0:\bar{N}|k}$ using (\ref{eq:substituted_tube}) with $v_{0:\bar{N}-1|k} = v_{0:\bar{N}-1}$ and $\S_k=\S$\;
            \If{$\X_{i|k}\subseteq\X$ \textnormal{\textbf{and}} $\kappa(\X_{i|k}) \oplus\{ v_{i}\} \subseteq \U$ \textnormal{for} $i=0,\ldots,\bar{N}-1$ \textnormal{\textbf{and}} $\X_{\bar{N}|k}\subseteq \X_T$}{
                \If{$J(v_{0:N-1},x_k,\S) \leqslant J^{*}_{k-1} - \ell(x_{k-1}, v^\ast_{0|k-1}) + \sigma^{2}$}{
                    $\feasible \gets \True$\;
                    \If{$\lvert \I \rvert < \lvert \I ^{\min} \rvert$}{
                        $\I^{\min} \gets \I$\;
                    }
                }
            }
        }
    }
}
$\bar{N}_{k} \gets \bar{N}$; $\S_{k} \gets \S(\I^{\min},\S_{k-1})$\;
\end{algorithm2e}}

The \texttt{while} loop (lines 4-16) implements (\ref{eq:horizon_shrink}) and (\ref{eq:feasible_indices}) with the modification that, for each value of $\bar{N}$, feasibility is checked for all candidate perturbation sequences in $\Vset_k$. This set must contain $\tilde{v}_{0:N-1|k}$ defined in (\ref{eq:tail_sequence}), but it may also contain other sequences that are likely to yield a feasible problem with as few predictor substitutions as possible. 
%(Appendix~\ref{appendix:feasibility_check} proposes one such  candidate perturbation sequence for the numerical example in Section~\ref{sec:example}). 
% attempts to recover the longest horizon yielding a feasible problem. Starting with the full horizon $N$, we perform a line search between the shifted previous optimal sequence $\tilde{v}_{k:k+N-1}$, and some other candidate trajectory denoted $v^{(c)}_{k:k+N-1} \in \Vset_{k}$. 
%
Since the predictor substitutions in $\S$ (computed in line 10) are only guaranteed to satisfy constraints if $v_{0:N-1}=\tilde{v}_{0:N-1|k}$, we perform an additional constraint check (lines 11-12).
The constraint (\ref{eq:cost_decrease_constraint}) on the predicted cost applies simultaneously to all stages, therefore its feasibility is checked (line 13) after computing $\S$.
If it is possible to generate a feasible problem using a candidate sequence in $\Vset_k$, the \texttt{for} loop (lines 7-16) determines the fewest predictor substitutions required and stores their indices in $\I^{\min}$  (line 16).
%check for state and input constraint satisfaction for the generated state tube. We then apply any necessary predictor substitutions and check the cost-decrease condition \eqref{eq:cost_decrease_constraint}. If satisfied, then the multi-step RHOCP with substitutions is known to be feasible.  
%
If no candidate in $\Vset_k$ generates a feasible problem, the algorithm executes the next iteration of the while loop with $\bar{N}$ reduced by $1$
until $\bar{N}=\max\{1,\bar{N}_{k-1}-1\}$. At this point, a feasible RHOCP is guaranteed to exist, possibly using substituted predictors at every step, in which case
$\lvert \I^{\min}\rvert = \bar{N}$ and 
$s_{i|k} = s_{i+1|k-1}$ for $i=1,\ldots, \bar{N}$.

Note that Assumption~\ref{as:terminal_ingredients} implies $\X_1(x,0)\subseteq \X_T$ for all $x\in\X_T$. If, in addition, $\X_N(x,0)\subseteq\X_T$ for all $x\in\X_T$, then Algorithm~\ref{alg:updates_fn_indices} recovers $\bar{N}_k=N$ whenever $\bar{N}_{k-1}=1$.

\section{Multi-step MPC with predictor substitution} \label{sec:multi-step_rhocp}
%The substituted predictors a further modification to the cost function. 
The multi-step RHOCP cost is defined at the $k$th time step in terms of a worst-case cost index analogous to~(\ref{eq:normal_mpc_cost}),
%\begin{equation}\label{eq:new_mpc_cost}
\[
  J(v_{0:N-1|k}, x_k, \S_k) = 
\! \sum_{i = 0}^{\bar{N}_k-1} 
\! \max_{x\in\X_{i|k}\!\!} \ell(x,v_{i|k})  + \max_{x\in \X_{\bar{N}_{k}|k}\!\!} \ell_T(x) .
\]%\end{equation}
For a given state $x_k$ and predictor indices $\S_k$, and with predictors $\X_{1:\bar{N}_k}$ defined by (\ref{eq:substituted_tube}), we define the RHOCP%
\begin{subequations}\label{eq:RHOCP_subs}%
\begin{equation}\label{eq:RHOCP_objective}
v^\ast_{0:N-1|k} = \argmin_{v_{0:N-1}}  \, J(v_{0:N-1}, x_k, \S_k)
\end{equation}
subject to $\X_{0|k} = \{x_k\}$ and $v_{i|k}=0$ for all $i\geq \bar{N}_k$,
and subject to, for $i=0,\ldots,\bar{N}_k-1$,%
%the multi-step dynamics \eqref{eq:multi-step_dynamics}, the substituted tube \eqref{eq:substituted_tube} with all past perturbations consistent with the actual applied values, and, for $i = 0, \dots, \bar{N}_{k}-1$:
\begin{align}%
\X_{i|k} &\subseteq \X
\label{eq:state_constraint_subs} \\
\kappa(\X_{i|k}) \oplus \{v_{i|k}\} & \subseteq \U 
\label{eq:input_constraint_subs} \\
\X_{\bar{N}_k|k} &\subseteq \X_{T} 
\label{eq:terminal_constraint_sub}
\end{align}
and subject to the cost upper-bound
\begin{equation}
J(v_{0:N-1}, x_k, \S_k) \leqslant J^\ast_{k-1} -  \ell(x_{k-1}, v^\ast_{0|k-1})+ \sigma^{2} .
\label{eq:cost_decrease_constraint_subs}
\end{equation}
\end{subequations}
%and subject to $v_{i|k} = 0$ for $i\geq \bar{N}_k$.
This problem is the basis of the multi-step MPC strategy.
% summarised in Algorithm~\ref{alg:mpc_algorithm}. 

% \begin{assum}\label{as:terminal_cost}
%     There exists a positive-semidefinite weighing matrix $P$ for the terminal cost $\|x\|^{2}_{P}$ such that \begin{equation}\label{eq:lyapunov_like_condition}
%     \|x\|^{2}_{P} - \|f_{1}(x, \kappa(x)) + w\|^{2}_{P} \succeq \|x\|^{2}_{Q}+\|\kappa(x)\|^{2}_{R} - \sigma^{2}
%     \end{equation}
%     $\forall x \in \X_{T}, \ w \in \W_{1}$
% \end{assum}

{\setlength{\algomargin}{1em}
\begin{algorithm2e}
\caption{Multi-step MPC} \label{alg:mpc_algorithm}
\DontPrintSemicolon
\SetKwInOut{Input}{Input}
\SetKwInOut{Output}{Output}
\Input{$x_{k}$, $v^\ast_{0:N-1|k-1}$, $J^\ast_{k-1}$, $\bar{N}_{k-1}$, $\S_{k-1}$}
\Output{$u_k$, $v^\ast_{0:N-1|k}$, $J^\ast_k$, $\bar{N}_{k}$, $\S_{k}$}
\tcc{Online computation at $k=0$:}
set $\bar{N}_0=N$, $S_0 = \{1,\ldots,N\}$, solve (\ref{eq:RHOCP_objective})-(\ref{eq:terminal_constraint_sub}), set
$J_0^\ast\gets J(v^\ast_{0:N-1|0},x_0,\S_0)$, $u_0\gets \kappa(x_0) + v^\ast_{0|0}$\;
\tcc{Online computation at $k=1,2,\ldots$:}
compute $\tilde{v}_{0:N-1|k}$ using $v^\ast_{0:N-1|k-1}$\;
compute $\bar{N}_k$ and $\S_k$ using Algorithm~\ref{alg:updates_fn_indices}\;
solve~(\ref{eq:RHOCP_subs}) for $v^\ast_{0:N-1|k}$ and set $J^\ast_k\gets J(v^\ast_{0:N-1|k},x_k,\S_k)$\;
set $u_k \gets \kappa(x_k) + v^\ast_{0|k}$
\end{algorithm2e}}

% \begin{algorithm}[h!]
% \caption{Multi-step MPC Algorithm With Recursive Feasibility} \label{alg:mpc_algorithm}
% \begin{algorithmic}[1]
% \REQUIRE $v^{*}_{0:k+N-1|0}$, $\bar{N}_{0}$, $\eta_{0}$,  $x_{0}$
% \STATE Compute online at time $k = 1, \ 2, \ \dots$
% \STATE Compute $\tilde{v}_{k:k+N-1}$ and $v^{(c)}_{k:k+N-1}$
% \STATE Compute $\eta_{k}$ and $\bar{N}_{k}$ using Algorithm \ref{alg:updates_fn_indices}
% \STATE Solve RHOCP with predictor substitutions \eqref{eq:RHOCP_subs}
% \STATE Apply control $u_{k} = \kappa(x_{k}) + v^{\,*}_{k|k}$
% \end{algorithmic}
% \end{algorithm}

The recursive feasibility and asymptotic performance guarantees of Algorithm~\ref{alg:mpc_algorithm} can be shown as follows.

\begin{thm} \label{theorem:recursive_feasibility}
If RHOCP \eqref{eq:RHOCP_subs} is feasible at initial time, $k=0$, then it remains feasible at all time steps, $k > 0$.
\end{thm}

\vspace{-\baselineskip}
\begin{pf}
Assume the RHOCP \eqref{eq:RHOCP_subs} is feasible at time step $k-1$ for $k>0$. Three possibilities exist for the RHOCP arising from Algorithm~\ref{alg:updates_fn_indices} at time $k$,
\\
%\mbox{}\quad \begin{tabular}{@{}l@{~~}l@{}}
\textbf{case 1:}~  $\bar{N}_{k} \geqslant \bar{N}_{k-1} > 1$ or $\bar{N}_{k} > \bar{N}_{k-1} = 1$,
\\
\textbf{case 2:}~ $\bar{N}_{k} = \bar{N}_{k-1}-1$, 
\\
\textbf{case 3:}~ $\bar{N}_{k} = \bar{N}_{k-1} = 1$.
%\end{tabular}

In case 1, Algorithm~\ref{alg:updates_fn_indices} terminates with $\textit{feasible} = \text{True}$, indicating that a perturbation sequence $v_{0:N-1}\in\Vset_k$ and predictor sequence $\S_k$ exist such that the state, input, terminal constraints, and cost bound \eqref{eq:state_constraint_subs}-\eqref{eq:cost_decrease_constraint_subs} are satisfied, thus guaranteeing feasibility of the RHOCP~(\ref{eq:RHOCP_subs}). 

In case 2, Algorithm~\ref{alg:updates_fn_indices} also terminates with $\textit{feasible} = \text{True}$.
%, thus guaranteeing RHOCP feasibility. 
This is because the time-shifted RHOCP solution, $\smash{\tilde{v}_{0:N-1|k}}\in\Vset_k$, trivially satisfies \eqref{eq:state_constraint_subs}-\eqref{eq:terminal_constraint_sub} 
with $\smash{\bar{N}_k}=\smash{\bar{N}_{k-1}}-1$ and with predictors substituted at every future time step, i.e.~$\smash{\S_k} = \{\smash{s_{2|k-1}},\ldots,\smash{s_{\bar{N}_{k-1}|k-1}}\}$
%i.e.~$s_{i|k} = s_{i+1|k-1}$, $i=1,\ldots,\bar{N}_k$
%
(no predictor is associated with stage $i=0$ since $\smash{x_k\in\X_{1|k}^\ast}$ due to Assumption~\ref{as:correct_predictor_sets}),
and because this choice of $\S_k$ gives 
\[
J(\tilde{v}_{0:N-1|k}, x_k, \S_k) 
- J^\ast_{k-1}+  \ell(x_{k-1}, v^\ast_{0|k-1}) \leqslant  0,
\]
so \eqref{eq:cost_decrease_constraint_subs} also holds. Hence a feasible RHOCP solution will be found by Algorithm~\ref{alg:updates_fn_indices} since $\tilde{v}_{0:N-1|k}\in\Vset_k$.
% Perhaps mention the case where this terminates with only a subset of substitutions -> same logic holds as for case 1. 

In case 3, we have $x_k\in\smash{\X_{1|k-1}^\ast}\subset\smash{\X_T}$. Therefore
Assumption~\ref{as:terminal_ingredients} implies \eqref{eq:state_constraint_subs}-\eqref{eq:terminal_constraint_sub} hold.
%with $\smash{\tilde{v}_{0:N-1|k}}= \{0,\ldots,0\}$. 
Moreover, \eqref{eq:cost_decrease_constraint_subs} holds since $\smash{\max_{x\in\X_{1|k-1}^\ast} \ell_T(x)} \geqslant \ell_T(x_k)$, and hence
\begin{align*}
%&J(\tilde{v}_{0:N-1|k}, x_k, \S_k) 
&J(\{0,\ldots,0\}, x_k, \S_k) 
- J^\ast_{k-1} + \ell(x_{k-1}, v^\ast_{0|k-1}) 
\\
& \qquad 
\leqslant
\max_{x\in\X_1(x_k,0)} \ell_T(x) + \ell(x_k,0) - \ell_T(x_k) \leqslant \sigma^2 ,
\end{align*}
where the final inequality is due to Assumption~\ref{as:terminal_ingredients}.
%
% % CAN THE HORIZON NEVER REACH MAYBE WE ADD A SPECIAL CASE THAT IN THIS SITUATION WE ALSO CHECK FOR V = 0?
% Furthermore, if horizon shrinks to $\bar{N}_{k} = 1$, case \textbf{4} is guaranteed not to occur at $k+1$ due to the terminal constraint, and the definition of the RPI terminal set. This guarantees that a problem can be constructed with a horizon length of at least 1, for which $v=0$ is a feasible, yet possibly suboptimal solution. 
\qed
\end{pf}

%\subsection{Closed-loop stability}\label{sec:stability}
\begin{thm}
The closed-loop system of (\ref{eq:plant}) with Algorithm~\ref{alg:mpc_algorithm} satisfies the constraints $x_k \in \X$ and $u_k \in \U$ for all $k\geqslant 0$ if the RHOCP~(\ref{eq:RHOCP_subs}) is feasible at $k=0$. 
Furthermore, the closed-loop system satisfies the bound
\[
\limsup_{t \to \infty} \frac{1}{t} \smash[t]{\sum^{t-1}_{\tau=0}} \bigl(\|x_{\tau}\|^{2}_{Q} + \|u_{\tau}\|^{2}_{R} \bigr) \leqslant \sigma^{2}.
\]
\end{thm}

\vspace{-1.5\baselineskip}
\begin{pf}\label{proof:stability}
Due to recursive feasibility and $\smash{\X_{0|k}} = \smash{\{x_k\}}$, constraints (\ref{eq:state_constraint_subs}) and (\ref{eq:input_constraint_subs}) at $i=0$ imply $x_k\in\X$ and $u_k = \kappa(x_k) + \smash{v^\ast_{0|k}}\in\U$ for all $k\geqslant 0$.
%Theorem~\ref{theorem:recursive_feasibility} implies that the RHOCP~(\ref{eq:RHOCP_subs}) is recursively feasible. Therefore, 
From~\eqref{eq:cost_decrease_constraint_subs} we have, for all $k>0$,
$\smash{J^\ast_{k}} \leqslant \smash{J^\ast_{k-1}} - \smash{\ell(x_{k-1},v^\ast_{0|k-1})} + \sigma^2 =
\smash{J^\ast_{k-1}} - \smash{\|x_{k-1}\|^{2}_{Q}} - \smash{\|u_{k-1}\|^{2}_{R}} + \sigma^{2}$,
and since $Q,R\succ 0$, it follows that $J^\ast_k$ is bounded for all $k>0$ and $\sigma^2$ upper-bounds the time-average stage cost.
\qed
\end{pf}

\begin{cor}
If, in Assumption~\ref{as:terminal_ingredients}, $\smash{\sigma^2} \leqslant \smash{\beta(\rho_0(\W_1))}$, where $\beta:\R_{\geqslant 0} \to \R_{\geqslant 0}$ is continuous, strictly increasing with $\beta(0)=0$, and $\smash{\rho_0({\W_1})}=\smash{\max_{w\in\W_1}\|w\|}$,
%is the radius of $\W$, 
then system~\eqref{eq:plant} with Algorithm~\ref{alg:mpc_algorithm} is ISpS on the set of initial states for which (\ref{eq:state_constraint_subs})-(\ref{eq:terminal_constraint_sub}) is feasible.
\end{cor}

\vspace{-0.5\baselineskip}
\begin{pf}
Under the conditions of the theorem, 
%$\ell_T(x)$ is an ISS Lyapunov function for the system $x_{k+1}=\bar{f}_1(x,0)+w$ for $x\in\X_T$, 
the RHOCP objective $\smash{J^\ast_{k}} = {J(v^\ast_{0:N-1|k},x_k,\S_k)}$ is an ISpS Lyapunov function due to~\eqref{eq:cost_decrease_constraint_subs}~\citep{limon:2009}.
\end{pf}

\section{Numerical example} \label{sec:example}
We demonstrate the multi-step MPC and predictor substitution algorithms using a simple example problem. While multi-step models are typically used in the context of data-driven control, to simplify analysis we consider a multi-step model constructed from repeated compositions of a single-step model.
Consider a nonlinear system in the form of (\ref{eq:plant}) with
\begin{equation}
\label{eq:example_system}
g(x, u) = Ax + Bu + \phi(x)
\end{equation}
where $x=[x_1 \ x_2]^\top\in\R^2$, $u\in\R^1$, and
\[
\phi(x) = \Bigl[\begin{smallmatrix}
    0 \\ e^{\alpha x_{2}}-1
\end{smallmatrix}\Bigr], 
\quad A = \Bigl[\begin{smallmatrix} 0.92 & 0.45 \\ 0.35 & 1.15 \end{smallmatrix}\Bigr], 
\quad B = \Bigl[\begin{smallmatrix} 0 \\ 1 \end{smallmatrix}\Bigr],
\]
with $\alpha = 0.25$. The disturbance in~(\ref{eq:plant}) is $d_k=w_k$ where $w = [w_{1} \ w_{2}]^\top$ represents a bounded additive process noise with $|w_{1}| \leqslant 0.08$ and $|w_{2}| \leqslant 0.12$. The system is subject to linear state and input constraints, $x\in\X = \{x : Hx \leqslant \mathbf{1}\}$ and $u\in\U = [-0.35 \, \ 0.35]$, where
\[
  H = \left[\begin{smallmatrix}
    ~~0.67 & ~~0.67 \\
    ~~0.67 & -1.00 \\
    -1.09 & -1.28 \\
    -1.37 & ~~1.00 \\
    -0.56 & ~~2.22
  \end{smallmatrix}\right], \quad
  \mathbf{1} = \left[\begin{smallmatrix} 1 \\ 1 \\ 1 \\ 1 \\ 1 \end{smallmatrix}\right].
\]
Given a future control sequence $u_{0:N-1|k}$ with horizon $N=10$, the following multi-step model is obtained 
%from compostions of the function $g$,
\begin{equation} \label{eq:example_multi-step}
        f_{i}(x_{k}, u_{0:i-1|k}) = g\bigl(\dots (g(x_{k}, u_{0|k}), \dots), u_{i-1|k}\bigr).
\end{equation}
% \begin{equation} \label{eq:example_multi-step}
%     \begin{gathered}
%         f_{1}(x_{k}, u_{0|k}) = g(x_{k}, u_{0|k}), \\
%         f_{2}(x_{k}, u_{0:1|k}) = g\bigl(g(x_{k}, u_{0|k}), u_{1|k}\bigr), \\
%         \vdots \\ 
%         f_{N}(x_{k}, u_{0:N-1|k}) = g\bigl(\dots (g(x_{k}, u_{0|k}), \dots), u_{N-1|k}\bigr),
%     \end{gathered}
% \end{equation}
% with $N = 10$.
We define the feedback law \eqref{eq:control_law} as $\kappa(x) = Kx$, so that the control law is given by $u = Kx + v$, where $K = [{-0.35} \,\, {-1.15}]$.
%, and define the closed-loop LDI matrices $\Phi^{(j)} = \hat{A}^{(j)} + \hat{B}^{(j)}K$. 
The closed-loop multi-step predictors $\bar{f}_{i}$ are as defined in \eqref{eq:tube_section}. 

The stage cost weights are $Q = I$ and $R = 1.5$. The terminal set is $\X_{T} = \{x : \|x\|^{2}_{P} \leqslant \gamma\}$, and terminal cost is $\ell_T(x)=\|x\|_P^2$, with  
\[
P = \Bigl[\begin{smallmatrix}
    15.54 & 0 \\ 0 & 147.70
\end{smallmatrix}\Bigr], \ \gamma = 2.29.
\]
Implementation details for this example 
%of Algorithms~\ref{alg:updates_fn_indices} and \ref{alg:mpc_algorithm} 
are provided in Appendix~A. 
The procedure for computing multi-step disturbance sets is given in Section~\ref{appendix:disturbance_sets}. The computation of the terminal set and feedback law is described in Section~\ref{appendix:terminal_set}. The methods for evaluating the infeasible index set $\I$ and computing the candidate perturbation set $\Vset_k$ are described in Section~\ref{appendix:feasibility_check}.
%and \ref{appendix:candidate_sequence}. 
Finally, a Sequential Convex Programming (SCP) approach for solving the RHOCP is discussed in Section~\ref{appendix:dc_solution}.

%To create a situation in which inconsistent multi-step predictors result in an infeasible problem, we 
We artificially inflate the disturbance sets for prediction time steps $i = 3$ and $i=4$ by a factor of two. This leads to sufficient inconsistency in predicted tubes at successive time steps 
%to create a situation where, without the use of the predictor substitution algorithm, the multi-step RHOCP becomes infeasible at  time steps $k = 1$ and $k=2$ despite being feasible at $k = 0$.
to make the multi-step RHOCP infeasible at $k = 1$ and $k=2$ despite being feasible at $k = 0$.
%
%With Algorithm~\ref{alg:updates_fn_indices}, the RHOCP~(\ref{eq:RHOCP_subs}) is feasible at $k=0$, but infeasible at $k=1$.
However, the predictor substitution procedure of Algorithm~\ref{alg:updates_fn_indices} constructs feasible problems using substitutions at $k=1$ and $k=2$. For all time steps $k\geqslant 3$, no further substitutions are required.
The controller is implemented using CVXPY~\citep{diamond2016cvxpy} and CLARABEL~\citep{goulart:2024}. Figure \ref{fig:simulations_no_shrinking} shows the resulting predicted and closed-loop trajectories.\\[-1ex]
%for an example where the 3- and 4-steps ahead disturbance sets are artificially enlarged.

\begin{figure}
    \centering
    \includegraphics[scale=0.7]{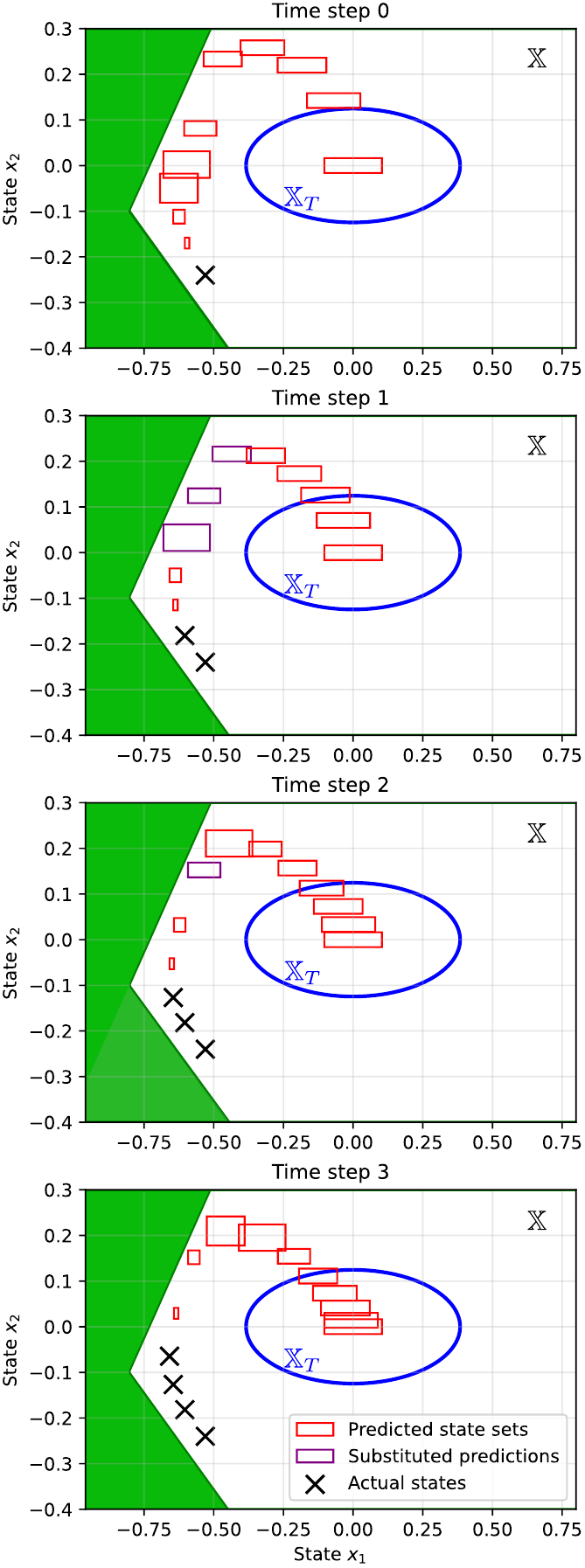}
    \caption{Predicted state tubes and closed-loop trajectories for several time steps, showing the predicted state sets generated using substituted predictors}
    \label{fig:simulations_no_shrinking}
\end{figure}

\section{Conclusions and future directions} \label{sec:conclusion}
This paper presents a method for guaranteeing recursive feasibility of robust multi-step MPC. By using predictor substitutions to avoid infeasibility of the RHOCP arising from inconsistencies in robust predictors at different time steps, the proposed method does not depend on a multi-rate formulation, and therefore allows the control sequence to be recomputed at each time step, given the latest state information. The algorithm is agnostic to the RHOCP solution method. A simple numerical example illustrates the approach. 
A potential limitation is that predictor substitution may cause suboptimality if the prediction horizon is reduced to maintain the format of one predictor per time step. Future work will relax this requirement by allowing arbitrary compositions of multi-step predictors, and will investigate alternative candidate control sequences for checking feasibility.
%some degree of conservatism, caused by the use of a simple line-search over candidate trajectories, which may result in unnecessary substitutions.

\vspace{0.5ex}
\bibliographystyle{agsm}
\bibliography{recursive_mpc_aut}

\vspace{0.5ex}
\appendix
\section{Details of the numerical example}
%\vspace{0ex}
\subsection{Robust multi-step disturbance sets}
\label{appendix:disturbance_sets}
The predictions of the model \eqref{eq:example_system} can be bounded using an incremental linear difference inclusion (LDI),
\begin{equation}\label{eq:LDI}
    g(x+e, u+Ke) - g(x,u) \in \mathrm{co}\{\Phi^{(1)}e, \dots, \Phi^{(m)}e\} ,
\end{equation}
for all $x,e\in\X$, $u\in\U$,
where $\Phi^{(j)} = \hat{A}^{(j)}+\hat{B}^{(j)}K$, $j=1,\ldots,m$, are the LDI vertices and $\mathrm{co}$ denotes the convex hull. 
For the system~(\ref{eq:example_system}), a suitable LDI can be defined with $m=2$ in terms of bounds on the derivative of $\phi(x) = [0 \ \ e^{\alpha x_2}-1]^\top$ for $x\in\X$.
%the maximum and minimum values of $x_{2}$ within $\X$. 
%
An outer bound on the $i$-step ahead robust disturbance set can then be computed using sums of single-step disturbance sets, i.e.~$\W_0 = \{0\}$, 
\begin{equation} \label{eq:disturbances}
\W_{i} = \D \oplus \mathrm{co}\{\Phi^{(1)}\W_{i-1} , \ldots, \Phi^{(m)}\W_{i-1} \} 
%\W_{i} = \max_{j\in \{1, \dots m\}} \bigoplus^{i-1}_{k = 0} (\Phi^{(j)})^{i -1 - k}\W_{1}
\end{equation} 
for $i = 1, \dots, N$.\\[-1ex]

\subsection{Robust terminal set and feedback law}
\label{appendix:terminal_set}
%We can locally stabilise the closed-loop system, and satisfy 
In Assumption~\ref{as:terminal_ingredients}, \eqref{eq:terminal_cost} holds if, for all $j$, we have
\[%\begin{equation}\label{eq:lyapunov_ldi}
    \|x\|^{2}_{P} - \|\Phi^{(j)}x + w\|^{2}_{P} \geqslant \|x\|^{2}_{Q + K^{T}RK} - \sigma^{2}
\]%\end{equation} 
for all $w \in \W_{1}=\mathrm{co}\{w^{(1)},\ldots,w^{(q)}\}$. Defining  $S = P^{-1}$ and $Y = KS$, this can be expressed in terms of the following linear matrix inequalities (LMIs): for all $j$ and~$i$,%
\begin{equation} \label{eq:schur_lmi}
    \left[\begin{smallmatrix}
        S & {0} & (\hat{A}^{(j)}S + \hat{B}^{(j)}Y)^\top & S & Y^\top \\
        {\star} & \sigma^{2} & w^{(i)\top} & {0} & {0} \\
        {\star} & {\star} & S & {0} & {0} \\
        {\star} & {\star} & {\star} & Q^{-1} &  {0} \\
        {\star} & {\star} & {\star} & {\star} & R^{-1}
    \end{smallmatrix}\right] \succeq {0}.
\end{equation}
Furthermore, to ensure that compositions of $g(\cdot )$ up to the prediction horizon remain convex, we must impose a further condition that the elements of $\Phi^{(j)}$ are nonnegative. This is achieved with the element-wise constraints:
\begin{equation}\label{eq:convexity_condition}
\hat{A}^{(j)}S + \hat{B}^{(j)}Y \geqslant {0} , 
\qquad 
S \succ 0, \ \mathrm{diagonal}.
\end{equation}
%and the constraint that $S$ is diagonal. 
To compute the terminal cost and feedback gain 
we solve the semidefinite program: $\mathrm{minimise}\ \sigma^{2}$
subject to \eqref{eq:schur_lmi}-\eqref{eq:convexity_condition}, and set
$P \gets S^{-1}$, $K \gets YP$.
%The terminal set is defined by the ellipsoidal set 
% \[
%     \X_{T} = \{ x : \|x\|^{2}_{P} \leqslant \gamma\}.
% \]
We compute the maximum terminal set scaling $\gamma$ satisfying $\X_T\subseteq \X$ and $K\X_T\subseteq \U$ and verify that
the RPI condition in (\ref{eq:multi-step_terminal}) holds since $\sigma^2$ is no greater than $\gamma$ multiplied by the smallest eigenvalue of $P^{-1/2}(Q + K^\top R K)P^{-1/2}$.\\[-1ex]
%
%if the resulting $\gamma$ satisfies $\gamma \geqslant \sigma^2/\sigma_{\min}(P^{-1/2}(Q + K^\top R K)P^{-1/2})$, where $\sigma_{\min}(\cdot)$ denotes the minimum singular value.
%
% \begin{equation}
%     \mathrm{maximise}\ \gamma
% \end{equation} subject to 
% \begin{subequations} \label{eq:gamma_constraints}
%     \begin{align}
%         \X_{T} \oplus \W_{1} &\subseteq \X \label{eq:example_state_constraint} \\
%         K\mathbb{X_{T}} \oplus \W_{1} &\subseteq \U \label{eq:example_input_constraint}
%     \end{align}
% \end{subequations} 
%where $\W_{i}$ is the steady-state effect the disturbance calculated according to \eqref{eq:disturbances}. 
%We note that this term only converges if the spectral radius of $\Phi^{(j)}$ for all $j = 1,\dots,m$.

\subsection{Infeasibility check and candidate control sequences}
\label{appendix:feasibility_check}
Since the model state in this example is two-dimensional, it is computationally
feasible to determine the infeasible set $\I$ in (\ref{eq:feasibility_check}) using the vertices of the disturbance sets $\smash{\W_1}, \ldots, \smash{\W_N}$.
% feasible to compute the feasibility checks for each vertex of the disturbance sets.
However, for higher-dimensional systems with polytopic disturbance sets  feasibility can be checked
% for higher-dimension systems, if the disturbances are represented as polytopes, it is possible to check this
via linear programming. Given a disturbance set $\W_i = \smash{\{w: H_{w}^{(i)} w \leqslant h_{w}^{(i)}\}}$, the condition that $\smash{\X_{i|k}} =\{\smash{ x^0_{i|k}}\}\oplus \smash{\W_i}\subseteq \X  = \smash{\{x : H x \leqslant \mathbf{1}\}}$ is equivalent to existence of a nonnegative matrix $\smash{\Lambda^{(i)}}\geqslant 0$ such that $\smash{\Lambda^{(i)}H_{w}^{(i)}} = H$ and $\smash{\Lambda^{(i)}h_{w}^{(i)}} \leqslant \mathbf{1} - \smash{ Hx^0_{i|k}}$.
%
%\subsection{Alternative } \label{appendix:candidate_sequence}

For this example, $\Vset_k$ is populated by $\smash{\tilde{v}_{0:N-1|k}}$, $\smash{ v^{\min}_{0:N-1|k}}$, and a set of linear interpolations between these two sequences, where $\smash{v^{\min}_{0:N-1|k}}$ is the sequence of perturbations which would result in $u_{i|k} = \min_{\U} u$ for all $i$ under the feedback law applied to a nominal state trajectory generated by the shifted sequence $\tilde{v}_{0:N-1|k}$.\\[-1ex]

\subsection{Solution of the multi-step RHOCP} \label{appendix:dc_solution}
To solve the RHOCP~(\ref{eq:RHOCP_subs}), we exploit the convexity properties of $g$, which allow efficient convex RHOCP approximations.
%, in this case with convex exponential cone and second order cone constraints.
%
Following \citet{DoffSotta:2022}, the predictors are defined by elementwise bounds,
\begin{equation}\label{eq:tube_DC}
\X_{i|k} = \{x : \underline{x}_{i|k} \leqslant x - x_{i|k}^0 \leqslant \overline{x}_{i|k}\} ,
\end{equation}
where $\smash{\underline{x}_{i|k}}$, $\smash{\overline{x}_{i|k}}$ are optimisation variables. Using the convex-concave procedure (CCP), at
each step $k$ we linearise around the nominal predicted state sequence $x^0_{0:N|k}$ defined by $x_{0|k} = x_k$ and, for $i=1,\ldots,N$
\[
x^0_{i|k} = \bar{f}_i(x_k, v^0_{0:i-1|k}),
\]
with $\smash{v^0_{0:N-1|k}=\tilde{v}_{0:N-1|k}}$. 
%It is possible to find tight bounds on the linearisation errors 
For this example, the CCP is guaranteed to converge 
since a convex function is lower-bounded by its Jacobian linearisation about any point in its domain. Let $\smash{v^\delta_{i|k}} = \smash{ v_{i|k} - v^0_{i|k}}$, then elementwise bounds at each prediction time step are given by
%\begin{equation} \label{eq:perturbation_bounds}
\begin{align*}
\overline{x}_{i|k} &\geqslant \bar{f}_{s}(x_{k-p}, \{v_{k-p:k-1},v_{0:i-1|k}\}) + \max_{w\in \W_{s}} w - x_{i|k}^0 
%\label{eq:upper_bound} 
\\
\underline{x}_{i|k} &\leqslant \bar{f}_{s}^\prime (x_{k-p}, \{v_{k-p:k-1},v^0_{0:i-1|k}\}) v^\delta_{0:i-1|k} + \smash{\min_{w\in \W_{s}} w}
%\label{eq:lower_bound}
\end{align*}    
%\end{equation}
for $i = 1, \dots, \smash{\bar{N}_{k}}$, where $s = \smash{ s_{i|k}}$, $p = \smash{ s_{i|k}} - i$
and $\smash{\bar{f}_{s}^{\prime}}$ is the Jacobian matrix of $\bar{f}_s$ with respect to the control sequence. 
%The sets in \eqref{eq:substituted_tube} are approximated by
% \begin{equation} \label{eq:tube_DC}
%     \X^{(s, DC)}_{i|k} = \{x^{nom}_{k+i}\} \oplus \mathbb{S}_{i|k}
% \end{equation} 
% where the sets $\mathbb{S}_{i|k} = \{s : \underline{s} \leqslant s_{i|k} \leqslant \overline{s}_{i|k}\}$ is a hyper-rectangle defined by the element-wise bounds in \eqref{eq:perturbation_bounds}. 
We replace the sets \eqref{eq:substituted_tube} in
% the cost \eqref{eq:new_mpc_cost} and constraints
\eqref{eq:RHOCP_subs} with \eqref{eq:tube_DC}. The RHOCP is then solved as a Sequential Convex Program, with the nominal trajectories updated at each iteration based on the previously solved convex problem.
\end{document}